\documentclass[a4paper,11pt]{amsart}

\usepackage[english]{my-shortcuts}
\usepackage{srcltx,algorithm,algorithmic}
\usepackage{a4wide}

\begin{document}

	\title[A method for column subset selection]
	{An elementary method for the problem of column subset selection in a rectangular matrix}
	\author{St\'ephane Chr\'etien and S\'ebastien Darses}

	\address{National Physical Laboratory\\ 
		Hampton road \\
		Teddington TW11 0LW, UK} 
	\email{stephane.chretien@npl.co.uk}
	
	\address{LATP, UMR 6632\\
		Universit\'e Aix-Marseille, Technop\^ole Ch\^{a}teau-Gombert\\
		39 rue Joliot Curie\\ 13453 Marseille Cedex 13, France}
	\email{sebastien.darses@univ-amu.fr}

	\maketitle
	
	
	\begin{abstract} The problem of extracting a well conditioned submatrix from any rectangular matrix (with e.g. normalized columns) has been a subject of extensive research with applications to rank revealing factorization, low stretch spanning trees, sparse solutions to least squares regression problems, and is also connected with problems in functional and harmonic analysis.
		Here, we provide a deterministic algorithm which extracts a submatrix $X_S$ from any matrix $X$ with guaranteed individual lower and upper bounds on each singular value of $X_S$. The proof of our main result is short and elementary.
	\end{abstract}
	
	
	{\bf keywords:} Column subset selection, Restricted Invertibility, 
	
	\section{Introduction}
	Let $X\in \R^{n\times p}$ be a matrix such that all columns of $X$ 
	have unit euclidean $\ell_2$-norm. We denote by $\|x\|_2$ the  $\ell_2$-norm of a vector $x$ and by $\|X\|$ (resp. $\|X\|_{HS}$) the associated operator norm (resp. the Hilbert-Schmidt norm). Let $X_T$ denote the submatrix of $X$ obtained by extracting the columns 
	of $X$ indexed by $T \subset \{1,\ldots,p\}$. For any real symmetric matrix $A$, let $\lb_k(A)$ denote the $k$-th eigenvalue of $A$, and we order the eigenvalues as $\lb_1(A)\ge \lb_2(A)\ge \cdots$. We also write $\lb_{\min}(A)$ (resp. $\lb_{\max}(A)$) 
	for the smallest (resp. largest) eigenvalue of $A$. We finally write $|S|$ for the size of a set $S$.
	
	The problem of well conditioned column selection that we condider here consists in finding the largest 
	subset of columns of $X$ such that the corresponding submatrix has all singular values in a prescribed interval 
	$[1-\epsilon,1+\epsilon]$. The one-sided problem of finding the largest possible $T$ such that $\lb_{\min}(X_T^tX_T)\ge 1-\epsilon$ 
	is called the Restricted Invertibility Problem and has a long history starting with the seminal work of Bourgain and Tzafriri 
	\cite{BourgainTzafriri:IJM87}. 
	Applications of such results are well known in the domain of harmonic analysis \cite{BourgainTzafriri:IJM87}. The study of 
	the condition number is also a subject of extensive study in statistics and 
	signal processing \cite{Tropp:CRAS08}.
	
	Here, we propose an elementary approach to this problem based on two simple ingredients:
	\begin{enumerate}
		\item Choosing recursively $y\in\cal V$, the set of remaining columns of $X$, verifying
		\bean
		Q(y) \le \frac{1}{|\cal V|} \sum_{x\in\cal V} Q(x),
		\eean
		where $Q$ is a relevant quantity depending on the previous chosen vectors;
		\item a well-known equation (sometimes called \emph{secular equation}) whose roots are the  eigenvalues of a square matrix after appending a row and a line.
	\end{enumerate}

	\subsection{Historical background}
	Concerning the Restricted Invertibility problem, Bourgain and Tzafriri \cite{BourgainTzafriri:IJM87} obtained the 
	following result for square matrices: 
	\begin{theo}[\cite{BourgainTzafriri:IJM87}]
		\label{BT87}
		Given a $p\times p$ matrix $X$ whose columns have unit $\ell_2$-norm, there exists $T\subset \{1,\ldots,p\}$ 
		with 
		$\d |T|\ge d\frac{p}{\|X\|^2}$ 
		such that  
		$C \le \lb_{\min}(X_T^tX_T) $, 
		where $d$ and $C$ are absolute constants.
	\end{theo}
	See also \cite{Tropp:StudiaMath08} for a simpler proof. 
	Vershynin \cite{Vershynin:IJM01} generalized Bourgain and Tzafriri's result to the case of rectangular matrices 
	and the estimate of $|T|$ was improved as follows. 
	\begin{theo}[\cite{Vershynin:IJM01}]
		\label{V01}
		Given a $n\times p$ matrix $X$ and letting $\wdt{X}$ be the matrix obtained from 
		$X$ by $\ell_2$-normalizing its columns. Then, for any $\epsilon\in (0,1)$, there exists $T\subset \{1,\ldots,p\}$ 
		with 
		\bean 
		|T|\ge (1-\epsilon) \frac{\|X\|_{HS}^2}{\|X\|^2}
		\eean 
		such that  
		$ C_1(\epsilon) \le \lb_{\min}(\wdt{X}_T^t\wdt{X}_T) \le \lb_{\max}(\wdt{X}_T^t\wdt{X}_T)\le C_2(\epsilon)$. 
	\end{theo}
	Recently, Spielman and Srivastava proposed in \cite{SpielmanSrivastava:IJM12} 
	a deterministic construction of $T$ which allows them to obtain the following
	result. 
	\begin{theo}[\cite{SpielmanSrivastava:IJM12}]
		Let $X$ be a $p\times p$ matrix and $\epsilon\in (0,1)$. Then there exists $T\subset \{1,\ldots,p\}$ 
		with  
		$\d |T|\ge (1-\epsilon)^2\frac{\|X\|_{HS}^2}{\|X\|^2}$ 
		such that  \ 
		$\d \epsilon^2 \frac{\|X\|^2}{p}  \le \lb_{\min}(X_T^tX_T)$. 
	\end{theo}
	The technique of proof relies on new constructions and inequalities which are thoroughly explained in the Bourbaki seminar of Naor \cite{Naor:Bourbaki12}.  Using these techniques, Youssef \cite{youssef} improved Vershynin's result as:
	\begin{theo}[\cite{youssef}]
		\label{youssef}
		Given a $n\times p$ matrix $X$ and letting $\wdt{X}$ be the matrix obtained from 
		$X$ by $\ell_2$-normalizing its columns. Then, for any $\epsilon\in (0,1)$, there exists $T\subset \{1,\ldots,p\}$ 
		with 
		$\d |T|\ge \frac{\e^2}{9} \frac{\|X\|_{HS}^2}{\|X\|^2}$ 
		such that  \ 
		$1-\epsilon  \le \lb_{\min}(\wdt{X}_T^t\wdt{X}_T) \le \lb_{\max}(\wdt{X}_T^t\wdt{X}_T) \le  1+\epsilon$. 
	\end{theo}

	\subsection{Our contribution}
	
	We provide a deterministic algorithm that extracts a submatrix $Y_r$ from the matrix $X$ with guaranteed individual lower and upper bounds on each singular value of $Y_r$.
	
	Consider the set of vectors $\cal V_0=\{x_1,\ldots,x_p\}$, where the $x_i$ are the columns of $X$. At step $r=1$, choose $y_1\in\cal V_0$. By induction, let us be given $y_1,\ldots,y_{r}$ at step $r$. Let $Y_r$ denote the matrix whose columns are $y_1,\ldots,y_r$ and let $v_k$ be an unit eigenvector of $Y_r^tY_r$ associated to $\lb_{k,r}:=\lambda_k(Y_r^tY_r)$. 
	
	We say that $u(\cdot,\cdot)$ satisfies the hypothesis (H) if $u$ verifies for $r\ge1$:
	\bea
	0\le u(k,r) & \le  & u(k+1,r+1), \quad k\in\{0,\cdots, r\}; \\
	0\le u(k+1,r) & \le & u(1,r)\ <\ u(0,r) \quad k\in\{1,\cdots, r-1\}.
	\eea
	
	We now introduce the "potential" associated to $u(\cdot,\cdot)$ satisfying (H): 
	\bean
	Q_r(x) & = & \sum_{k=1}^r  \frac{( v_k^t Y_r^t x)^2}{u(0,r)-u(k,r)}, \quad x\in \cal V_0.
	\eean 
	We then choose $y_{r+1}\in\cal V_r:=\{x_1,\ldots,x_p\} \setminus \{y_1,\ldots,y_{r}\}$ so that 
	\bea
	\label{taxto}
	Q_r(y_{r+1})  & \le & \frac1{p-r} \sum_{x\in \cal V_r}
	Q_r(x)
	= \frac1{p-r}  \sum_{k=1}^r  \frac{\sum_{x\in \cal V_r}( v_k^t Y_r^t x)^2}{u(0,r)-u(k,r)}.
	\eea

	The following result, for which we propose a short and elementary proof,  
	gives a control on {\em all} singular values in the column selection problem.

	\begin{theo} \label{main}
		Let $u$ satisfies Hypothesis (H). Set $R\le p/2$. Then, we can extract from $X$ some submatrices $Y_r$ such that
		for all $r$ and $k$ with $1\le k\le r\le R$, we have 
		\bea \label{hr}
		1- \delta_R\ u(r-k+1,r) \sqrt{\lb_{1,r}} \ \le \ \lb_{k,r} \ \le \  1+ \delta_R\  u(k,r) \sqrt{\lb_{1,r}},
		\eea 
		where
		\bea \label{delta}
		\delta_R & = & \sqrt{\frac{2\Vert X \Vert^2}{p}   \sup_{1\le r\le R}\sum_{k=1}^r  \frac{u(0,r)^{-1}}{u(0,r)-u(k,r)}
		}.
		\eea
		In particular,
		\bean
		\lb_{1,r}  &  \le & 1+ 2\delta_R\ u(1,r).
		\eean 
		
	\end{theo}

	\section{Proof of Theorem \ref{main}}

	\subsection{Suitable choice of the extracted vectors}
	Consider the set of vectors $\cal V_0=\{x_1,\ldots,x_p\}$. At step $1$, choose $y_1\in\cal V_0$. By induction, let us be given $y_1,\ldots,y_{r}$ at step $r$. Let $Y_r$ denote the matrix whose columns are $y_1,\ldots,y_r$ and let $v_k$ be an unit eigenvector of $Y_r^tY_r$ associated to $\lb_{k,r}:=\lambda_k(Y_r^tY_r)$. 
	Let us choose $y_{r+1}\in\cal V_r:=\{x_1,\ldots,x_p\} \setminus \{y_1,\ldots,y_{r}\}$ so that 
	\beq
	\label{taxto}
	\sum_{k=1}^r  \frac{( v_k^t Y_r^ty_{r+1})^2}{u(0,r)-u(k,r)}  \le  \frac1{p-r} \sum_{x\in \cal V_r}
	\sum_{k=1}^r  \frac{( v_k^t Y_r^t x)^2}{u(0,r)-u(k,r)} 
	= \frac1{p-r}  \sum_{k=1}^r  \frac{\sum_{x\in \cal V_r}( v_k^t Y_r^t x)^2}{u(0,r)-u(k,r)}.
	\eeq
	\begin{lemm} \label{yr}
		For all $r\ge 1$, $y_{r+1}$ verifies
		\bean
		\sum_{k=1}^r  \frac{( v_k^t Y_r^ty_{r+1})^2}{u(0,r)-u(k,r)} & \le & \frac{\lb_{1,r} \|X \|^2}{p-r} \sup_{1\le j\le r} \sum_{k=1}^j  \frac{1}{u(0,j)-u(k,j)}
		.
		\eean
	\end{lemm}
	\begin{proof}
		Let $X_r$ be the matrix whose columns are the $x\in \cal V_r$, i.e. $X_rX_r^t=\sum_{x\in \cal V_r}xx^t$. Then
		\bean \label{totu}
		\sum_{x\in \cal V_r} (v_k^t Y_r^tx)^2  =  {\rm Tr} \left(Y_r v_k v_k^t Y_r^t X_rX_r^t    \right)
		\le 
		{\rm Tr}(Y_r v_k v_k^t Y_r^t) \Vert X_rX_r^t \Vert \le  \lb_{k,r} \Vert X \Vert^2,
		\eean
		which yields the conclusion by plugging in into (\ref{taxto}) since $\lb_{k,r}\le \lb_{1,r}$.
	\end{proof}

	\subsection{Controlling the individual eigenvalues}
	
	It is clear that (\ref{hr}) holds for $r=1$ since then, 1 is the only singular value because the columns are supposed to be normalized.
	
	Assume the induction hypothesis $(H_r)$: for all $k$ with $1\le k\le r<R$, (\ref{hr}) holds.
	
	Let us then show that $(H_{r+1})$ holds. 
	By Cauchy interlacing theorem, we have 
	\bean
	\lb_{k+1,r+1} & \le & \lb_{k,r}, \quad 1\le k\le r \\
	\lb_{k+1,r+1} & \ge & \lb_{k+1,r}, \quad 0\le k\le r-1.
	\eean 
	We then deduce, due to the induction hypothesis $(H_r)$ and Assumption (H),
	\bea
	\lb_{k+1,r+1} & \le & 1+ \delta_R u(k,r)\sqrt{\lb_{1,r}} \le  1+ \delta_R u(k+1,r+1)\sqrt{\lb_{1,r+1}}, \quad 1\le k\le r,\\
	\lb_{k+1,r+1} &\ge & 1- \delta_R u(r-k,r)\sqrt{\lb_{1,r}} \nonumber\\
	& \ge & 1- \delta_R u(r+1-(k+1)+1,r+1)\sqrt{\lb_{1,r+1}}, \quad 0\le k\le r-1.
	\eea
	
	It remains to obtain the upper estimate for $\lb_{1,r+1}$ and the lower one for $\lb_{r+1,r+1}$. We write
	\bea \label{add}
	Y_{r+1}^tY_{r+1} & = & 
	\left[
	\begin{array}{c}
		y_{r+1}^t \\
		Y_r^t
	\end{array}
	\right]
	\left[
	\begin{array}{cc}
		y_{r+1} & Y_r
	\end{array}
	\right]
	=
	\left[
	\begin{array}{cc}
		1 & y_{r+1}^tY_r \\
		Y_r^ty_{r+1} & Y_r^tY_r
	\end{array}
	\right],
	\eea
	and it is well known that the eigenvalues of $Y_{r+1}^tY_{r+1}$ are the zeros of the secular equation:
	\bea
	q(\lb) :=  1-\lb+\sum_{k=1}^r  \frac{(v_k^tY_{r}^ty_{r+1})^2}{\lb-\lb_{k,r}} \ = \ 0.
	\eea
	
	We first estimate $\lb_{1,r+1}$ which is the greatest zero of $q$, and assume for contradiction that 
	\bea
	\lb_{1,r+1}> 1+ \delta_R u(0,r) \sqrt{\lb_{1,r}}. 
	\label{assup}
	\eea 
	From $(H_r)$, we then obtain that for $\lb \ge 1+ \delta_R u(0,r) \sqrt{\lb_{1,r}}$,
	\bean
	q(\lb) \le 1-\lb +\frac{1}{\delta_R\sqrt{\lb_{1,r}}} \sum_{k=1}^r  \frac{(v_k^tY_{r}^ty_{r+1})^2}{u(0,r)-u(k,r)}:=g(\lb).
	\eean
	Let $\lb^0$ be the zero of $g$. We have $g(\lb_{1,r+1})\ge q(\lb_{1,r+1})=0=g(\lb^0)$. But $g$ is decreasing, so
	\bean 
	\lb_{1,r+1} \le \lb^0 = 1+ \frac{1}{\delta_R\sqrt{\lb_{1,r}}} \sum_{k=1}^r  \frac{(v_k^tY_{r}^ty_{r+1})^2}{u(0,r)-u(k,r)}.
	\eean 
	Thus, using Lemma \ref{yr}, the equality (\ref{delta}) and noting that $r\le p/2$, we can write:
	\bea 
	\lb_{1,r+1} \le  1+ \frac{2}{\delta_R}  \frac{\sqrt{\lb_{1,r}}\Vert X \Vert^2}{p}\sum_{k=1}^r  \frac{1}{u(0,r)-u(k,r)} \ \le \
	1+ \delta_R u(0,r) \sqrt{\lb_{1,r}},
	\eea 
	which yields a contradiction with the inequality (\ref{assup}). Thus, we have 
	\beq
	\lb_{1,r+1} \le 1+ \delta_R u(0,r) \sqrt{\lb_{1,r}} \le  1+ \delta_R u(1,r+1) \sqrt{\lb_{1,r+1}}. 
	\eeq
	This shows that the upper bound in $(H_{r+1})$ holds.
	
	Finally, to estimate $\lb_{r+1,r+1}$ which is the smallest zero of $q$, we write
	\bean
	q(\lb) \ge 1-\lb -\frac{1}{\delta_R\sqrt{\lb_{1,r}}} \sum_{k=1}^r  \frac{(v_k^tY_{r}^ty_{r+1})^2}{u(0,r)-u(k,r)}:=\wdt g(\lb).
	\eean
	By means of the same reasonning as above, we show that the lower bound in $(H_{r+1})$ holds.

	\subsection{Controlling the greatest eigenvalue}
	
	Set $\mu_{1,r}=\lb_{1,r}-1\ge 0$. 
	
	Since $u(1,r)\le u(1,R)\le u(0,R)$, we can write
	\bean
	\mu_{1,r}  & \le & \delta_R \sqrt{\mu_{1,r} +1}.
	\eean
	Hence, using that $x\le A\sqrt{1+x}$ implies $x\le 2 A$, we reach the upper estimate for $\lb_{1,r}$.
	
	This concludes the proof of Theorem \ref{main}.

	\section{Two simple examples and an open question}

	Let us choose $u(k,r)=\frac{2r-k}{\sqrt{r}}$.
	Using $(r+1)(2r-k)^2\le r(2r+1-k)^2$ and $(r+1)(r+k)^2\le r(r+1+k)^2$, we thus deduce 
	that $u$ verifies Hypothesis (H). Applying Theorem \ref{main}, we obtain that we can extract a submatrix with $R$ columns and  $\lb_{1,R}  \le  1+ \e$,
	provided that 
	\bean
	R\log R & \le & \frac{\epsilon^2 }{8} \frac{p}{\Vert X \Vert^2},
	\eean
	which is a slightly weaker bound than the one known from \cite{BourgainTzafriri:IJM87}. 
	
	One can also verify that $u(k,r)= \sqrt{r-k}$ satisfies Hypothesis (H) and yields a similar bound.
	
	An open question is then to know whether there exists a function $u$ satisfying Hypothesis (H) and allowing to reach the optimal bound known in the Bourgain Tzafriri theorem \cite{BourgainTzafriri:IJM87} via our new algorithm.
	
	\bibliographystyle{amsplain}
	\bibliography{database}

\begin{thebibliography}{10}
		\bibitem{BourgainTzafriri:IJM87} 
		Bourgain, J. and Tzafriri, L., Invertibility of "large'' submatrices with applications to the geometry of Banach spaces and harmonic analysis. Israel J. Math. 57 (1987), no. 2, 137--224. 
		
		\bibitem{Naor:Bourbaki12} Naor, A., Sparse quadratic forms and their geometric applications [following Batson, Spielman and Srivastava]. S\'eminaire Bourbaki: Vol. 2010/2011. Expos\'es 1027--1042. Ast\'erisque No. 348 (2012), Exp. No. 1033, viii, 189--217. 
		
		\bibitem{SpielmanSrivastava:IJM12} Spielman, D. A. and Srivastava, N., An elementary proof of the restricted invertibility theorem. Israel J. Math. 190 (2012), 83--91.
		
		\bibitem{Tropp:StudiaMath08} Tropp, J., The random paving property for uniformly bounded matrices, 
		Studia Math., vol. 185, no. 1, pp. 67--82, 2008.
		
		\bibitem{Tropp:CRAS08} Tropp, J., Norms of random submatrices and sparse approximation. 
		C. R. Acad. Sci. Paris, Ser. I (2008), Vol. 346, pp. 1271-1274.
		
		\bibitem{Vershynin:IJM01} Vershynin, R., John's decompositions: selecting a large part. Israel J. Math. 122 (2001), 253--277. 
		
		\bibitem{youssef} Youssef, P. A note on column subset selection.
		Int. Math. Res. Not. IMRN  2014, no. 23, 6431--6447.
	\end{thebibliography}

\end{document}